\documentclass[12pt,english]{article}
\usepackage{babel}
\usepackage[cp1250]{inputenc}
\usepackage[T1]{fontenc}
\usepackage{amsfonts}
\usepackage{amsmath}
\usepackage{amsthm}
\usepackage{hyperref}
\usepackage{enumerate}
\usepackage{graphicx}
\usepackage{pictex}
\usepackage{color}

\newcommand{\rj}{r_{1}}
\newcommand{\rd}{r_{2}}
\newcommand{\kf}{\widetilde{K}}
\newcommand{\g}{\Gamma}

\newcommand{\gf}{\widetilde{\Gamma}}
\newcommand{\hd}{\hspace{0.2cm}}

\newcommand{\no}{\noindent}
\newcommand{\la}{\lambda}
\newcommand{\sk}{\overline{S}}
\newcommand{\skk}{\overline{{\sk}}}

\newcommand{\lap}{\Delta}
\newcommand{\rh}{\varrho}

\newcommand{\te}{\theta}

\newcommand{\eqq}[2]{\begin{equation}  #1  \label{#2} \end{equation}    }
\newtheorem{remark}{\textbf{Remark}}[section]

\newtheorem{theorem}{\textbf{Theorem}}[section]

\newcommand{\ep}{\varepsilon}

\newcommand{\rr}{\mathbb{R}}

\newcommand{\m}[1]{\mbox{ #1}}

\def\bn{\bf n}

\def\kf{R_2}

\newtheorem{definition}{\textbf{Definition}}[section]
\newtheorem{lemma}{\textbf{Lemma}}[section]

\newcommand{\skke}{\skk_{\ep}}
\newcommand{\ue}{U_{\ep}}
\newcommand{\qe}{Q_{\ep}}
\newcommand{\op}{O(\ep)}
\newcommand{\lue}{L^{2}(\ue)}
\newcommand{\we}{w_{\ep}}
\newcommand{\wje}{w_{1,\ep}}

\newcommand*{\sgn}{\mathop{\mathrm{sgn}}}

\newcommand{\Oe}{\Omega_{\ep}}
\newcommand{\Oex}{\Oe\cap\{ |x|<\ep\}}
\newcommand{\ioe}{\int_{\Oex}}
\newcommand{\ipoe}{\int_{\partial(\Oex)}}
\newcommand{\igxe}{\int_{\g_{\ep}\cap \{|x|<\ep \} }}

\begin{document}
\title{\bf Fine singularity analysis of solutions to the Laplace equation: Berg's effect}
\author{Adam Kubica\\Faculty of Mathematics and Information Science\\
Warsaw University of Technology\\
ul. Koszykowa 75, 00-662 Warsaw, POLAND\\
{\tt A.Kubica@mini.pw.edu.pl}\\
Piotr Rybka\\Faculty of Mathematics, Informatics and  Mechanics\\
The University of Warsaw\\
ul. Banacha 2, 02-097 Warsaw, POLAND\\{\tt rybka@mimuw.edu.pl}}

\maketitle
\date{}

\abstract{We study Berg's effect on special domains. This effect is understood as monotonicity of a harmonic function (with respect to the distance from the center of a flat part of the boundary) restricted to the boundary. The harmonic function must satisfy piecewise constant Neumann boundary conditions. We show that Berg's effect is a rare and fragile phenomenon.}


\section{Introduction}
We would like to apply the regularity results, we obtained in
\cite{KR}, to a study of the boundary behavior of  harmonic
functions. The motivation for this work comes from the observation made
by Berg, \cite{berg}, on crystals growing from a dilute solution. His conclusion
may be expressed as follows, if the process is quasistatic, i.e. slow,
and the crystal facets do not break nor bend, then the concentration, $c$,
restricted to any facet, is an increasing function of the distance
from the center of the facet. Since the process is quasistatic, then
the concentration is a harmonic function outside of the crystal
and the steady growth of facets implies that its normal derivative is
constant on each facet.

A weaker version of Berg's effect is well-known in the physics literature (see
e.g. \cite[\S 3 eq.(6)]{Ne}): the concentration  at a facet
center is smallest and  its value at the edge of the
facet is the largest.

There have been a few  attempts to establish this rigorously, e.g.
\cite{Se}, the most recent one is \cite{GR}. The
argument in \cite{Se} is based on explicit formulas. The authors of
\cite{GR} tried to establish Berg's effect analyzing the boundary
behavior of harmonic functions.
However,  the
proof of
regularity, \cite[Lemma 1.]{GR}  as noted in \cite{KR} contains a flaw. Thus the
issue reopens. We will not resolve the problem studied in \cite{GR},
but address a simpler one. Namely, we consider a planar domain
$\Omega$, defined as follows, for positive numbers $r_{1}, r_{2}$   and
$\la_{0}>1$ we set
\eqq{
R_{1}=(-r_{1}, r_{1})\times (-r_{2}, r_{2}), \hd R_{2}=\la_{0}R_{1}, \hd \Omega= R_{2} \setminus \overline{R_{1}}.
}{defR}
We define $\g = \partial R_{1}$, $\gf = \partial R_{2}$. We
consider the following problem,
\begin{equation}
\left\{
\begin{array}{ll}
 \Delta u =0& \hbox{in }\Omega,\\
 u =0& \hbox{on } \gf,\\
 \frac{\partial u}{\partial\bn} = u_n& \hbox{on } \g.
\end{array}
\right.
\label{main}
\end{equation}
As we noted, Neumann boundary condition is piecewise constant, i.e.
\eqq{
 u_n =
\left\{
\begin{array}{ll}
 a & \hbox{for } |x_2| = r_2,\\
 b & \hbox{for } |x_1| = r_1.
\end{array} \right.
}{defun}
where $a$ and $b$ are given numbers.

We have shown in \cite{KR} that for any pair $r_1,$ $r_2>0$ defining
$\Omega$, there is a pair $(a,b)$ (hence any other pair $(\lambda a,
\lambda b)$, where $\lambda\neq 0$ will do) such that the unique weak
solution to (\ref{main}) with data (\ref{defun}), $u$, is in
$C^1(\bar\Omega)$.

In the
present setting studying  Berg's effect amounts to investigating the
positivity of derivatives with respect to $x$ or $y$ of a weak
solution to (\ref{main}) with data (\ref{defun}), on the
appropriate parts of the boundary of the inner rectangle.
\begin{definition}
We assume that $a,b>0$. We shall say  Berg's effect holds for $(\Omega,a,b)$ if for $u$ a weak solution of (\ref{main}), (\ref{defun}) with boundary conditions  $u_{n}$, it is true that
\eqq{xu_{x}\leq 0 \m{ on } \g_{1}, \hd \hd yu_{y}\leq 0 \m{ on } \g_{4}.}{ca}
\label{bergdef}
\end{definition}
\no Here $\g_{1}$ and $\g_{4}$ are perpendicular sides of the inner rectangle (see next section for notation). 
We show (see  Theorem \ref{bergcj} for the proof):
\begin{theorem}
 Berg's effect holds for $(\Omega,a,b)$  if and only if $u,$ a weak solution to (\ref{main}) with data (\ref{defun})  is regular, i.e. $u\in C^1(\bar\Omega)$. 
\end{theorem}
The $C^1$-regularity of weak solutions is rather an exception not a rule, thus
Berg's effect is rare. It is so, because it holds  only for certain data $a$ and  
$b$.  We may say that the effect is not stable with respect to the perturbation of boundary data.
\begin{theorem}
Assume that $r_{1}=r_{2}$. Then Berg's effect holds for $(\Omega,a,a)$, where $a>0$. More generally, for any $r_{1},r_{2}>0$ and $a>0$ there exists a unique positive number  $b$ such that Berg's effect  for $(\Omega,a,b)$ holds.
\label{corouni}
\end{theorem}

In addition, perturbations of the domain destroy the effect. We can express it as follows.
For $r_{1},r_{2}>0$ and for $\ep>0,$ we set
\eqq{
R_{1,\varepsilon}=(-r_{1}-\ep, r_{1}+ \ep)\times (-r_{2}, r_{2}), \hd R_{2,\ep}= \lambda_{0}R_{1, \ep}, \hd \Oe=R_{2,\ep}\setminus \overline{R_{1,\ep}},
}{defoep}
Then,  $\Oe$ is a perturbation of $\Omega_{0}=\Omega$ and we have
\begin{theorem}
We assume that $a$ is  a given positive number. We take $b$ is a unique positive number such that Berg's effect holds for $(\Omega,a,b)$. Then, there exists $\ep_{0}>0$ such that for each $\ep \in (0, \ep_{0})$ Berg's effect fails for $(\Oe,a,b)$, where $\Oe$ is defined above.
\label{stab}
\end{theorem}

There is also another question left open. Namely, what is the
relation between $a/b$ (or $b/a$) and the proportions of the rectangle
$r_2/r_1$. The current analysis does not give any clue besides the
obvious statement: if $r_1 = r_2$, then $a$ must be equal to $b$. We must
write that despite its simplicity the problem has not been treated in
the literature and the available studies of singularities of solutions
to elliptic problems do not permit to make the desirable conclusion (at
least in an obvious way). Possibly, tools used in the monographs \cite{dauge},
\cite{Grisvard}, \cite{kondratiew}, \cite{kozlov}
were too general, while the methods used for numerical studies
\cite{ACS1}, \cite{ACS2}, \cite{CH} were
not able to capture the phenomenon we talk about.

\bigskip\noindent{\bf Notation}\\
Now, we recall the notation introduced in \cite{KR}.  We shall write,
\begin{eqnarray*}
 &&\g
=\partial R_1=\g_{1} \cup\g_{2} \cup\g_{3} \cup\g_{4}\cup S_{1} \cup
S_{2} \cup S_{3} \cup S_{4},\\
&&\gf =\partial \kf=\gf_{1}
\cup\gf_{2} \cup\gf_{3} \cup\gf_{4}\cup \widetilde{S}_{1} \cup
\widetilde{S}_{2} \cup \widetilde{S}_{3} \cup \widetilde{S}_{4} ,
\end{eqnarray*}
where $\g_{i}$, $\gf_{i}$  are sides of rectangles and
$S_{i}$, $\widetilde{S}_{i}$ are their
vertexes, $i=1,\ldots,4$. To be precise, we set $\g_{1}=\{(t,\rd): \hd t\in
(-\rj,\rj) \}$, (resp.  $\widetilde\Gamma_1=\{(t,\rj): \hd t\in
(-\rd,\rd) \}$). The remaining sides of $R_1$ (resp. $R_2$), i.e. $\Gamma_2$, $\Gamma_3$, $\Gamma_4$,
(resp. $\widetilde\Gamma_j$, $j=2,3,4$)
are visited
counterclockwise. 
We also set $S_{i}=\overline{\g_{i}} \cap \overline{\g_{i+1}}$, with the understanding
that $\Gamma_{4+1} = \Gamma_1$ and we define
$\widetilde S_j$ in the same manner. The distance from vertex $S_i$ is denoted by $\rh_{i}$.

For $i=2,4$, we set $\te_{i}$ to be the angle measured at $S_i$
from $\g_{i}$ to $\g_{i+1}$. At the same time for $i=1,3$, we set
$\te_{i}$ to be the angle measured from  $\g_{i+1}$ to $\g_{i}$.
Furthermore, by $\skk$ we denote the dual singular solution for $\Omega$ given by definition 2.1 \cite{KR}.

\section{Dual singular solutions}

We proved in  \cite{KR} that there are  five possible forms of the zero level set of $\skk$. In this section, we shall show that in fact only one of them is attained. For this purpose we introduce an additional notation. For $\ep >0, $ we set
\[
R_{1, \ep }=(-r_{1}- \ep , r_{1}+ \ep )\times (-r_{2}, r_{2}), \hd R_{2, \ep }=\la_{0}R_{1, \ep }, \hd \Omega_{\ep }= R_{2, \ep } \setminus \overline{R_{1, \ep }}.
\]
Let $\skk $ ($\skke$ resp.) be a dual  singular  solution constructed in definition 2.1 \cite{KR} for the domain $\Omega$ ($\Omega_{\ep}$ resp.). First, we shall show the continuity of dual singular solutions with respect to small domain perturbations. More precisely, we have:

\begin{lemma}
Let $\skk$ and $\skke$ be as above. Then, for any $K $ compact subset of $\overline{\Omega} $ which does not contain vertices  $\{S_{1}, S_{2}, S_{3}, S_{4}\}$, the function $\skke$ converges uniformly to $\skk$ on $K$.
\label{la}
\end{lemma}

\no For the set $\Omega_{\ep}$, we adopt the same notation as for
$\Omega$, i.e. $\partial \Omega_{\ep}= \Gamma_{\ep}\cup
\widetilde{\Gamma}_{\ep}$, the vertices of $\Gamma_{\ep}$ are $S_{i,
  \ep}$, the vertices of  $\widetilde{\Gamma}_{\ep}$ are denoted by
$\widetilde{S}_{i,\ep}$, the sides of $\Gamma_{\ep}$ are
$\Gamma_{i,\ep}$ and the sides of $\widetilde{\Gamma}_{\ep}$ are
$\widetilde{\Gamma}_{i,\ep}$, $i=1,2,3,4.$

\begin{proof}
We denote $\ue= \Omega \cap \Omega_{\ep}$. Thus, the boundary of $\ue $ consists of two parts: $\g_{\ep}$ and $\gf$. On $\ue$,  we set
\[
\qe (x,y )= \left\{
\begin{array}{cll}
\skk(x- \ep \sgn{x}, y  ) & \mbox{ for } & |x|>\ep, \\
\skk(0, y) & \mbox{ for } & |x|\leq\ep.  \\
\end{array}
 \right.
\]
Then $\qe$ is continuous on $x=\ep$ and $\nabla \qe$ is also continuous on $x=\ep$, because $\skk_{x}(0, y )=0$ by symmetry with respect to $y$-axis. Next, $\frac{\partial \qe}{\partial n }=0$ on $\Gamma_{\ep}$ and $\qe=0$ on $\widetilde{\Gamma}_{1,\ep}$ and $\widetilde{\Gamma}_{3,\ep}$. Further, $\skk$ is harmonic in $\Omega$, bounded near $\widetilde{\Gamma}$ and $\skk=0$ on $\gf$, hence
\eqq{\qe= \op \hd \m{ on } \hd  \gf_{2} \m{and } \gf_{4}.  }{a}
Furthermore, we shall show that for a constant $c_{0}$, the estimate
\eqq{|\lap \qe |\leq c_{0}\chi_{ \{|x|<\ep \}} \m{ on } \ue, }{b}
holds, independently on $\ep$, provided $\ep$ is small enough.  Indeed, $\qe $ is harmonic on the set $\{|x| >\ep   \}$ and $\lap \qe (x,y)= \skk_{yy}(0,y)$ on $\{|x|<\ep \}$. Since $\skk$ is harmonic and smooth away from vertices (Proposition 2.1 \cite{KR}), hence we can bound $\skk_{yy}(0,y)$  (estimates near $\g$ and $\gf$ are obtained after applying an appropriate reflection with respect to the boundary).

\no Finally, from $\| \skk \|_{L^{2}(\Omega)}=1$, we can deduce that
\eqq{\| \qe \|_{\lue} = 1+ \op.}{c}
Now, let us suppose that $\we\in H^{1}(\ue)$ is a unique weak solution of the following problem,
\[
\left\{
\begin{array}{rlcll}
\lap \we &=& \lap \qe & \m{ in } & \ue, \\
\frac{\partial \we }{\partial n }& =& 0 & \m{ on } & \g_{\ep}, \\
\we& =& \qe & \m{ on } & \gf, \\
\end{array}
\right.
\]
Thus, from (\ref{a}) and (\ref{b}), we deduce that
\eqq{\| \we \|_{\lue}= \op.}{d}
We set
\eqq{v_{\ep}= \frac{\qe - \we}{\| \qe - \we \|_{\lue}}. }{e}
Then, $v_{\ep}$ is a dual singular  solution for domain $\ue$, hence by Corollary 2.1 \cite{KR}, $v_{\ep} $ satisfies (11)-(14) \cite{KR} for domain $\ue$.

\no Let $\wje\in H^{1}(\ue)$ be a unique weak solution of the following problem
\[
\left\{
\begin{array}{rlcll}
\lap \wje &=& 0 & \m{ in } & \ue \\
\frac{\partial \wje }{\partial n }& =& 0 & \m{ on } & \g_{\ep} \\
\wje& =& \skke & \m{ on } & \gf \\
\end{array}
\right.
\]
By definition, function $\skke$ vanishes on $\gf_{\ep}$, hence proceeding as earlier,  we get ${\skke}_{|\gf} = \op$. Thus, standard estimate leads to
\eqq{ \| \wje \|_{\lue}= \op. }{g}
We put
\eqq{v_{1,\ep}= \frac{\skke - \wje}{\| \skke - \wje \|_{\lue}}. }{h}
Then $v_{1,\ep}$   satisfies (11)-(14) \cite{KR} for domain $\ue$, therefore from Corollary 2.1 \cite{KR} we get $v_{\ep}= \pm v_{1, \ep}$. However,  the singular parts of $v_{\ep}$ ($ v_{1, \ep}$ resp.) come from $\skk$ ($\skke$ resp.), thus $v_{\ep}=  v_{1, \ep}$, i.e.
\[
\qe - \we= \frac{\| \qe - \we \|_{\lue}}{\| \skke - \wje \|_{\lue}} (\skke - \wje),
\]
If we use (\ref{c}), (\ref{d}), (\ref{g}) and $\| \skke \|_{\lue}= 1+ \op$, then
\eqq{\qe - \skke= \we- \wje + \op (\skke - \wje). }{i}
From conditions  (\ref{d}) and (\ref{g}), by using  the standard regularity argument, we obtain  that for any fixed $K$, satisfying  the Lemma assumption, $\we$ and 
$\wje$ 
converge uniformly to zero on $K$. Next, for $K$ as above, ${\skke}_{|K}$ is uniformly bounded with respect to $\ep$, provided $\ep $ is small enough (see Proposition 2.1 \cite{KR}). Therefore,
\eqq{\qe - \skke \rightarrow 0 \m{ uniformly on } K.}{jj}
Finally, we notice that $\skk$ is continuous on 
$K$, hence $\qe  \rightarrow  \skk $ uniformly on $K$. Therefore, 
the claim of the Lemma follows. 
\end{proof}

\begin{remark}
The claim of Lemma~\ref{la} is also true if $\skke$ is the dual singular
solution defined for domain $\Omega_{\ep}$, when $\Omega_{\ep}$ is a
small perturbation of $\Omega$ in both directions, $x$-axis and $y$-
axis. In this case we have to extend $\skke$ across the boundary. It
could be done by odd reflection with respect to $\gf_{\ep}$ and by
even reflection with respect to $\g_{\ep}$. The proof requires only
minor modifications.      \label{cola}
\end{remark}

Let us fix $\la_{0}>1$. Then, each $(r_{1}, r_{2})\in \rr_{+}^{2}$ defines
domain $\Omega= \Omega(r_{1}, r_{2})$ for which we construct a dual
singular solution $\skk=\skk(r_{1}, r_{2})$. In the proof of 
\cite[Lemma~2.9]{KR}, we indicate that zero level sets of $\skk$  may have five
possible shapes. The $k$-th possibility is illustrated by $k$-th
figure in \cite{KR}, $k=1,2,3,4$ and the fifth possibility corresponds
to the situation when the vertex $S_{1}$ and the side $\g_{4}$ are
connected by a zero level set.    We divide the set $\rr^{2}_{+}$ into
three parts, $A_{1}$, $A_{2}$, $A_{3}$ in the following way:\\
$(r_{1},
r_{2})\in A_{1}$ if for $\skk(r_{1}, r_{2})$ the first or the forth
possibility holds;\\ $(r_{1}, r_{2})\in A_{2}$ if for $\skk(r_{1},
r_{2})$ the second   or the fifth possibility holds;\\
$(r_{1},
r_{2})\in A_{3}$ if for $\skk(r_{1}, r_{2})$ the third possibility
holds. Then,
\begin{equation}\label{rem1}
\rr^{2}_{+}= A_{1}\cup A_{2}\cup A_{3}.
\end{equation}
In other words, $(r_{1},r_{2})\in A_{1}$ if a dual singular solution for $\Omega(r_{1},r_{2})$ is positive in a neighborhood of $\gf$,   $(r_{1},r_{2})\in A_{2}$ if it is negative  in a neighborhood of $\gf$ and $(r_{1},r_{2})\in A_{3}$ if $\skk$ changes sign in any neighborhood of $\gf$ . We will show that Lemma~\ref{la} implies that the structure of zero level sets of $\skk$ remains unchanged under a small perturbation of the domain. More precisely,

\begin{lemma}
The sets $A_{i}$ are open.
\label{lb}
\end{lemma}

\begin{proof}
Let $(r_{1}, r_{2})\in A_{1}$ and $\skk(r_{1}, r_{2})$ be a
corresponding solution for $\Omega=\Omega(r_{1}, r_{2})$. For each
$\ep=(\ep_{1}, \ep_{2})$, we have  $\skke$ defined for domain
$\Omega_{\ep}$, where $\Omega_{\ep}= R_{2, \ep} \setminus
\overline{R_{1,\ep}} $, $R_{1, \ep }= (- r_{1}- \ep_{1},  r_{1}+
\ep_{1})\times (- r_{2}- \ep_{2},  r_{2}+ \ep_{2})$, $R_{2, \ep }=
\la_{0} R_{1, \ep}$. We shall show that for $\ep_{i}$ close to zero
the corresponding solution $\skke$ has the zero level set as in the
first or forth case. Firstly, we note that $\skk$ is
positive in a neighborhood of $\gf$. We denote the set
$R_{2}\setminus (\la_{0}-\delta)R_{1}$ by $K$, where $\delta$ is so small 
that $\skk_{|(\la_{0}-\delta)\partial R_{2}}>0$. For this set $K$, we
use Lemma~\ref{la} (more precisely Remark~\ref{cola}). Then $\skke
\rightarrow \skk$ on $K$ and hence ${\skke}_{|(\la_{0}-\delta)\partial
  R_{2}}>0$ for $\ep_{i}$ small enough. It means that for such
$\ep=(\ep_{1}, \ep_{2})$, the first or the fourth possibility holds,
i.e. $( r_{1}+ \ep_{1},  r_{2}+ \ep_{2})\in A_{1}$, in other words set 
$A_{1}$ is open.

For set $A_{2}$, we may proceed in a similar manner, because in this case the corresponding solution $\skk$ is negative in a neighborhood of $\gf$ and we can argue as earlier.

It remains to show that $A_{3}$ is open. First, we notice that if $(r_{1},r_{2})\in A_{3}$, then the corresponding solution $\skk$ has at least two points $x_{1}, x_{2}\in \gf$, such that
\eqq{\frac{\partial \skk}{\partial n}(x_{1})>0, \hd \frac{\partial \skk}{\partial n}(x_{2})<0. }{jjj}
Indeed,  in this case the zero level set consists of four analytic curves, each of them connects vertex $S_{i}$ with $\gf$. These curves divide $\Omega$ into four regions. In two of them  $\skk$ is positive and in the other two regions $\skk$ is negative. Hence, in each of these regions, $\skk$ attains its supremum or infimum on $\gf$, because $\skk_{|\gf}=0$.  Then, by Hopf Lemma, we get (\ref{jjj}).

Further, from Lemma \ref{la} and Remark~\ref{cola}, we have uniform convergence of $\skke$ to $\skk$ on compact subsets, which do not contain any vertex.  But $\skke$ are harmonic, thus their derivatives also converge uniformly to derivatives of $\skk$ on such compact subsets. Thus, \[
\frac{\partial \skke}{\partial n}(x_{1})>\frac{1}{2}\frac{\partial \skk}{\partial n}(x_{1})>0, \hd \frac{\partial \skke}{\partial n}(x_{2})<\frac{1}{2}\frac{\partial \skk}{\partial n}(x_{2})<0,
\]
for $\ep$ small enough. Finally, we conclude that
\[
\frac{\partial \skke}{\partial n}(x_{1, \ep})>0, \hd  \frac{\partial \skke}{\partial n}(x_{2,\ep})>0,
\]
for $x_{1, \ep}, x_{2, \ep}\in \gf_{\ep}$, provided $\ep$ is sufficiently small. 

Thus, the above inequalities imply that $\skke$ is negative in a neighborhood of $x_{1, \ep}\in \gf_{\ep}$ and positive in a neighborhood of $x_{2, \ep}\in \gf_{\ep}$. This is so, because ${\skke}_{|\gf_{\ep}}=0$.  By the continuity of $\skke$ in $\Omega_{\ep}$, $\skke$ attains zero on every curve connecting these neighborhoods. As a result, the zero level set of $\skke$ connects $\g_{\ep}$ and $\gf_{\ep}$, provided that $\ep$ is sufficiently small. Hence, 
$A_{3}$ is open.
\end{proof}

\begin{theorem}
We assume that $r_{1},r_{2}>0$, $\la_{0}>1$ and $\Omega$ is defined by (\ref{defR}). Let $\skk$ be the dual singular solution  constructed in definition 2.1 \cite{KR} for the domain $\Omega$.
Then, $W_{0}$, the zero level set of $\skk$ consists of four analytic curves. Each of them connects one vertex $S_{i}$ with the outer part of the boundary $\gf$. In particular, $\skk_{|\g_{1}}<0$ and $\skk_{|\g_{2}}>0$, hence $\int_{\g_{1}}\skk<0$ and $\int_{\g_{2}}\skk>0$.
\label{structure}
\end{theorem}

\begin{proof}
From (\ref{rem1}) and Lemma~\ref{lb}, we deduce that exactly one
from the sets $A_{i}$ is not empty. Lemma~2.10 \cite{KR} shows that
$A_3$ is not empty, because $(r_{1},r_{1})\in A_{3}$. Thus $\rr^{2}_{+}=A_{3}$. It means that for each $(r_{1},r_{2})\in \rr^{2}_{+}$, the structure of the level set is the same as in the third possibility, i.e. $S_{i}$ and $\gf$ are connected by the zero level set.
In order to prove the remaining statements, we note that $\skk$ is continuous on $\g_{i}$ and $\inf\limits_{\g_{1}}\skk=-\infty$, \hd   $\sup\limits_{\g_{2}}\skk=\infty$. If   $\max\limits_{\g_{1}} \skk=\skk(x_{0},r_{2}) \geq 0$, then $\skk$ restricted to the set $\{(x,y)\in \Omega: \hd \skk(x,y)<0 \}$ has a maximum on $\g_{1}$, thus  by  Hopf Lemma, the outer normal derivative at $(x_{0},r_{2})$ is positive, which contradicts the definition of $\skk$. A similar argument works for $\g_{2}$.

\end{proof}

\section{Berg's effect}

Berg's effect is related to the boundary behavior of  harmonic function, $u$, in a domain of the form $\Omega \setminus P$, where $P$ is a polygon, and  the normal derivative of $u$ is constant on each side of $P$. Here, we study only harmonic functions, which are solutions to (\ref{main}), (\ref{defun}).


Our main observation is that Berg's effect holds if and only if $u$ is a regular solution to (\ref{main}), (\ref{defun}).

\begin{theorem}
Berg's effect holds for $(\Omega,a,b)$ if and only if the corresponding weak solution of (\ref{main}), (\ref{defun})  is in $C^{1}(\overline{\Omega})$.
\label{bergcj}
\end{theorem}

\begin{proof}
Suppose that $u $ is a weak solution of (\ref{main}) and $u\not \in C^{1}(\overline{\Omega})$. Then, by Corollary~2.2 \cite{KR} $a\int_{\g_{1}}\skk+ b\int{\g_{2}}\skk \not = 0$, hence there is a singular part of solution $u$ and by Proposition~2.1 \cite{KR}, we have
\[
u = c_{2,1}\rh_{1}^{\frac{2}{3}}\cos{\frac{2}{3}\theta_{1}}+ w \m{ on } B(S_{1}, \delta),
\]
for a $\delta>0$, where $w\in C^{1}(\overline{\Omega}\cap
B(S_{1},\delta))$ and $c_{2,1}\not =0$ (see Section 1 for the
definitions of $\rh_{1}$ and $\theta_{1}$). Due to the boundary conditions
$w_{x}(S_{1})<0$ and $w_{y}(S_{1})<0$, hence by the continuity of $\nabla
w$, we get ${w_{x}}_{|\g_{1}\cap B(S_{1}, \delta_{1})}<0$ and
${w_{y}}_{|\g_{4}\cap B(S_{1}, \delta_{1})}<0$ for a
$\delta_{1}>0$. On the other hand, ${(\rh_{1}^{\frac{2}{3}}
  \cos{\frac{2}{3}\theta_{1}}
  )_{x}}_{|\g_{1}}=(-(r_{1}-x)^{\frac{2}{3}})_{x}=
\frac{2}{3}(r_{1}-x)^{-\frac{1}{3}}\nearrow \infty$ if $x\rightarrow
r_{1}^{-}$ and ${(\rh_{1}^{\frac{2}{3}} \cos{\frac{2}{3}\theta_{1}}
  )_{y}}_{|\g_{4}}=((r_{2}-y)^{\frac{2}{3}})_{y}=
-\frac{2}{3}(r_{1}-y)^{-\frac{1}{3}}\searrow -\infty$ if $y\rightarrow
r_{2}^{-}$. Thus, for a $\delta_{2}>0$, we have
${xu_{x}}_{|\g_{1}\cap B(S_{1}, \delta_{2})}>0$ if $c_{2,1}>0$ and
${yu_{y}}_{|\g_{4}\cap B(S_{1}, \delta_{2})}>0$ if
$c_{2,1}<0$. Thus, Berg's effect does not hold for $(\Omega,a,b)$.

In order to prove the other implication, we suppose that $u$ is a  weak solution of
(\ref{main}), which  belongs to $C^{1}(\overline{\Omega})$ but Berg's
effect does not hold for $(\Omega,a,b)$, i.e. $u_{x}(x,r_{2})>0$ for
an $x \in (0,r_{1})$ or $u_{y}(r_{1}, y)>0$ for a $y\in
(0,r_{2})$. We shall show that the first possibility leads to a
contradiction (the reasoning in the other case is the same). For this
purpose we denote  the set $\Omega \cap \{(x,y): \hd
x>0 \}$ by $\Omega_{+}$ and we shall consider $u_{x}$ in $\Omega_{+}$. Then, $u_{x}$ is
harmonic in $\Omega_{+}$ and continuous on
$\overline{\Omega_{+}}$. First, we note that $u$ is positive in
$\Omega$. Indeed, by definition $u_{|\gf}=0$, thus if $u$ were
negative at a point of $\Omega$, then $u$ would have a negative
minimum on $\g$, but then we would get a contradiction with Hopf Lemma,
because by definition  $\frac{\partial u}{\partial n}_{|\g}>0$. Hence,
$u$ is positive in $\Omega$ and from boundary condition $u_{|\gf}=0$,
we deduce that ${u_{x}}_{|\gf_{4}} \leq 0$. Next, by condition
$u_{|\gf}=0$, we get ${u_{x}}_{|\gf_{1} \cup \gf_{3}}=0$ and by
the symmetry of $u$ with respect to $\{x=0 \}$, we obtain $u_{x}=0$ on
$\Omega \cap \{x=0 \}$. Finally, by definition (see (\ref{defun})),
${u_{x}}_{|\g_{4}}=-b$. Thus, if $u_{x}(x,r_{2})>0$ for an $x\in
(0,r_{1})$, then  $u_{x}$ restricted to $\Omega_{+}$ admits a positive
maximum, which is necessarily  located on $\g_{1}$, say at
$(x_{0},r_{1})\in \g_{1}$. Then, by Hopf Lemma, we deduce that
$\frac{\partial u_{x}}{\partial n}(x_{0},r_{1})>0 $, but on the other
hand using boundary condition (\ref{defun}) we get  $\frac{\partial u_{x}}{\partial
  n}(x_{0},r_{1})= \frac{\partial}{\partial x} \frac{\partial
  u}{\partial n}(x_{0},r_{1})=\frac{\partial}{\partial x} a=0$, which 
yields a contradiction. Thus, $u_{x}$ is non positive on the boundary of
$\Omega_{+}$, hence it is negative in $\Omega_{+}$. In particular, ${u_{x}}_{|\Omega_{+}}\leq 0 $, hence ${xu_{x}}_{|\g_{1}}\leq 0$. The proof that ${yu_{y}}_{|\g_{4}}\leq 0 $ is analogous.
\end{proof}

\begin{proof}[Proof of Theorem~\ref{corouni}]
If $r_{1}=r_{2}$, then applying Theorem~1.2 \cite{KR}, we get $u\in C^{1}(\overline{\Omega})$ and from  Theorem~\ref{bergcj} we deduce that Berg's effect holds for $(\Omega,1,1)$. In the general case, from Theorem~\ref{structure} we deduce that numbers $\alpha_{1}=\int_{\g_{1}}\skk$ and $\beta_{1}=\int_{\g_{4}}\skk$ from Theorem~1.1 \cite{KR} are not zero, hence the claim follows from Theorem~1.1 \cite{KR} and Theorem~\ref{bergcj}.
\end{proof}

In the above considerations, we analyze the stability of Berg's effect under a perturbation of boundary conditions. Below, we shall investigate its stability 
under a domain perturbation.

\begin{proof}[Proof of Theorem~\ref{stab}]
Recall that domain $\Omega_{\ep}$ is defined by (\ref{defoep}) and positive number  $b$ is given by Corollary~\ref{corouni}. Denote by $u$ a weak solution of the problem
\eqq{\left\{ \begin{array}{rclll}
\lap u &=& 0 &  \m{ in } & \Omega, \\
\frac{\partial u }{\partial n } &=& u_{n} & \m{ on } & \g, \\
u &=&0 & \m{  on } &  \gf, \\
\end{array}
\right.
}{ta}
where $u_{n}$ is defined in (\ref{defun}). Then, from Theorem~\ref{bergcj}  and  \cite[Proposition~2.1]{KR}, we have  $u\in H^{2}(\Omega)\cap C^{1}(\overline{\Omega})$. We set,
\[
u_{\ep} (x,y)= \left\{ \begin{array}{cll}  u(x-\ep \sgn{x},y) & \m{ for } & \ep<|x|\leq r_{1}+\ep \\
u(0,y) & \m{ for } &  |x|\leq \ep \\ \end{array}\right.
\]
From the symmetry  of  $u$ with respect to  $\{x=0 \}$ we get  $u_{x}(0,y)=0$, hence  $u_{\ep}\in  H^{2}(\Oe)\cap C^{1}(\overline{\Omega_{\ep}})$ and
\[
\lap u_{\ep} (x,y)= \left\{ \begin{array}{cll}  0 & \m{ for  } & \ep<|x|\leq r_{1}+\ep \\
u_{yy}(0,y) & \m{ for  } &  |x|\leq \ep \\ \end{array}\right.
\]
We set  $f_{\ep}= \lap u_{\ep}$. Then, we notice that $u_{\ep}=0$ on $\gf_{\ep}$ and  $\frac{\partial u_{\ep} }{\partial n }= u_{n,\ep} $   on $ \g_{\ep}$, where
\eqq{
 u_{n,\ep} =
\left\{
\begin{array}{ll}
 a & \hbox{for } |x_2| = r_2,\\
 b & \hbox{for } |x_1| = r_1+\ep.
\end{array} \right.
}{defune}
\no Function  $f_{\ep}$ belongs to $  L^{2}(\Oe)$. Hence, there exists a unique  $w_{\ep}\in H^{1}(\Oe)$, a solution of the problem
\eqq{\left\{ \begin{array}{rll}
\lap w_{\ep} = f_{\ep} &  \m{ in } & \Oe, \\
\frac{\partial w_{\ep} }{\partial n }= 0 & \m{ on } & \g_{\ep}, \\
w_{\ep} =0 & \m{  on } &  \gf_{\ep}. \\
\end{array}
\right.
}{tb}
We denote $v_{\ep}= u_{\ep}- w_{\ep}$. Then, $v_{\ep}$ satisfies
\eqq{\left\{ \begin{array}{rrlll}
\lap v_{\ep} &=&0  &  \m{ in } & \Oe, \\
\frac{\partial v_{\ep} }{\partial n } &=& u_{n,\ep} & \m{ on } & \g_{\ep}, \\
v_{\ep} &=&0 & \m{  on } &  \gf_{\ep}. \\
\end{array}
\right.
}{tc}
By Theorem~\ref{bergcj}, Berg's effect holds for $(\Omega_{\ep},a,b)$ if and only if $v_{\ep} $ is in $C^{1}(\overline{\Omega_{\ep}})$.
Therefore, we shall investigate the smoothness of $v_{\ep}$. From Lemma~2.1 and Proposition~2.2 in \cite{KR} we have the decomposition of solution of  (\ref{tc}),
\[
v_{\ep}= v_{r\ep}+ c_{\ep}\sk_{\ep},
\]
where  $v_{r\ep}\in H^{2}(\Oe)\cap C^{1}(\overline{\Oe})$, $\sk_{\ep}\in H^{1}(\Oe)\setminus H^{2}(\Oe)$, $c_{\ep}= - 2a\int_{\g_{1,\ep}}\skk_{\ep}- 2b\int_{\g_{2,\ep}}\skk_{\ep}$.

On the other hand, using a standard argument  (proof of Lemma~2.1 \cite{KR}), we get
\[
w_{\ep}= w_{r\ep}+ \widetilde{c}_{\ep}\sk_{\ep},
\]
where $w_{r\ep}\in H^{2}(\Oe)\cap C^{1}(\overline{\Oe})$ and  $\widetilde{c}_{\ep}=  \int_{\Oe}\skk_{\ep}f_{\ep}$. Thus,
\[
\underbrace{u_{\ep}}_{\in H^{2}(\Oe)}= v_{\ep}+ w_{\ep}= \underbrace{(v_{r\ep}+ w_{r\ep})}_{\in H^{2}(\Oe)}+ (c_{\ep}+ \widetilde{c}_{\ep})\underbrace{\sk_{\ep}}_{\not \in H^{2}(\Oe)}.
\]
Hence,  $c_{\ep}= - \widetilde{c}_{\ep}$ and we obtain
\eqq{\int_{\Oe} \skk_{\ep}f_{\ep}=2a\int_{\g_{1,\ep}} \skk_{\ep}+2b\int_{\g_{2,\ep}}\skk_{\ep}.  }{td}
We would like to show that $v_{\ep}$ is not in $C^{1}(\overline{\Omega})$ for small $\ep$. For this purpose, it is enough to show that the  left hand side of (\ref{td}) is not zero.   First, we notice that using Lemma~\ref{la}, we conclude that $\ep \mapsto \int_{\Oe} \skk_{\ep}f_{\ep}$ is continuous at $0$. 

Thus, the proof will be finished, if
\eqq{\lim_{\ep\rightarrow 0^{+}}\frac{1}{\ep} \int_{\Oe} \skk_{\ep}f_{\ep} \not =0. }{nna}
Using Lemma~\ref{la} again we get
\[
\lim_{\ep\rightarrow 0^{+}}\frac{1}{\ep} \int_{\Oe} \skk_{\ep}f_{\ep} dxdy =\lim_{\ep\rightarrow 0^{+}}\frac{1}{\ep}  \int_{\Oex} \skk_{\ep} u_{yy}(0,y) dxdy
\]
\[
= 2\int_{r_{2}}^{\lambda_{0}r_{2}} \skk (0,y)u_{yy}(0,y)dy= - 2\int_{r_{2}}^{\lambda_{0}r_{2}} \skk (0,y)u_{xx}(0,y)dy.
\]
By Theorem~\ref{structure} we deduce $\skk (0,y)<0$ for $y\in (r_{2}, \lambda_{0}r_{2})$. Finally, in the second part of  the proof of Theorem~\ref{bergcj} we deduce that $u_{x}$ is negative in  $\Omega_{+}= \Omega\cap \{ x>0 \}$ and by symmetry $u_{x}=0$ on $\Omega \cap \{x=0 \}$. Hence the  function  ${u_{x}}_{|\Omega_{+}}$ has its maximum on $\Omega \cap \{x=0 \}$. Therefore from Hopf Lemma w deduce that $u_{xx}(0,y)<0$ for $y\in  (r_{2}, \lambda_{0}r_{2})$, hence
\[
- 2\int_{r_{2}}^{\lambda_{0}r_{2}} \skk (0,y)u_{xx}(0,y)dy <0,
\]
 and the proof is finished.
\end{proof}

\subsection*{Acknowledgment}
Both authors were partially supported by NCN through 2011/01/B/ST1/01197
grant.

\end{document}